\documentclass[12pt,reqno,a4paper]{amspaper}
\usepackage{mathdef}
\usepackage{natbib}
\newcommand*\tageq{\refstepcounter{equation}\tag{\theequation}}
\usepackage[utf8]{inputenc}
\usepackage{comment}
\usepackage{setspace}
\usepackage{graphicx}
\usepackage{multirow}
\usepackage{lscape}
\usepackage{afterpage}
\usepackage{caption}
\usepackage{subcaption}
\usepackage[section]{placeins}
\usepackage{flafter}
\usepackage{rotating}
\usepackage{eufrak}
\usepackage{mathabx}
\usepackage{mathrsfs}
\usepackage{bbm}
\usepackage{url}
\usepackage{todonotes}
\DeclareMathAlphabet{\mathpzc}{OT1}{pzc}{m}{it}
\newcommand{\opnorm}[1]{\vvvert{#1}\vvvert_{\mathcal{L}}}
\newcommand{\snorm}[1]{\vvvert{#1}\vvvert}

\usepackage{enumitem}
\usepackage{color}
\AtBeginDocument{%
  \addtolength\abovedisplayskip{-0.045\baselineskip}%
  \addtolength\belowdisplayskip{-0.045\baselineskip}%
}

\makeatletter
\newcommand{\leqnomode}{\tagsleft@true\let\veqno\@@leqno}
\newcommand{\reqnomode}{\tagsleft@false\let\veqno\@@eqno}
\makeatother
\newcommand{\E}{\mean}
\newcommand{\Var}{\var}

\newcommand{\Hspace}{\mathbbm{H}}

\newcommand{\F}{\mathcal{F}}

\newcommand{\Fm}{\mathscr{F}}
\newcommand{\Tr}{\trace}
\newcommand{\Bcom}{\mathcal{L}^+_{\infty}}
\newcommand{\lo}{\mathcal{L}}

\newcommand{\Bol}{\mathcal{L}(H)}

\newcommand{\inprod}[2]{\langle #1, #2 \rangle}

\begin{document}
\reqnomode
\title[A note on Herglotz's theorem for time series on function spaces]%
{A note on Herglotz's theorem for time series on function spaces}
\author{Anne van Delft}
\thanks{Ruhr-Universit{\"a}t Bochum, Fakult{\"a}t f{\"u}r Mathematik, 44780 Bochum, Germany}
\thanks{{\em E-mail address:} anne.vandelft@rub.de  (A.~van Delft)} 
\author{Michael Eichler}
\thanks{Department of Quantitative Economics,
Maastricht University, P.O.~Box 616, 6200 MD Maastricht, The Netherlands}
\thanks{{\em E-mail address:} m.eichler@maastrichtuniversity.nl (M.~Eichler)}

\dedicatory{{\upshape July 10, 2019}}
\begin{abstract}
In this article, we prove Herglotz's theorem for Hilbert-valued time series. This requires the notion of an operator-valued measure, which we shall make precise for our setting. Herglotz's theorem for functional time series allows to generalize existing results that are central to frequency domain analysis on the function space. In particular, we use this result to prove the existence of a functional Cram{\'e}r representation of a large class of processes, including those with jumps in the spectral distribution and long-memory processes.  We furthermore obtain an optimal finite dimensional reduction of the time series under weaker assumptions than available in the literature. The results of this paper therefore enable Fourier analysis for processes of which the spectral density operator does not necessarily exist.

\noindent
{{\itshape Keywords:} Functional data analysis, spectral analysis, time series}
\smallskip 

\noindent 
{{\em 2010 Mathematics Subject Classification}. Primary: 60G10;
Secondary: 62M15.}
\end{abstract}

\maketitle

\section{introduction}

As a result of a surge in data storage techniques, many data sets can be viewed as being sampled continuously on their domain of definition. It is therefore natural to think of the data points as being objects and embed them into an appropriate mathematical space that accounts for the particular properties and structure of the space. The development of meaningful statistical treatment of these objects is known as functional data analysis and views each random element as a point in a function space. Not surprisingly, this has become an active field of research in recent years. If the random functions can be considered an ordered collection $\{X_t\}_{t \in \znum}$ we call this collection dependent functional data or a functional time series. The function space where each $X_t$ takes its values is usually assumed to be the Hilbert space $L^2([0,1])$, in which case we can parametrize our functions $\tau \mapsto X_t(\tau), \tau \in [0,1]$. 

While the literature on classical time series finds its origin in harmonic analysis, the literature on its functional counterpart started in the time domain. The frequency domain arises however quite naturally in the analysis of dependent functional data. The second order dependence structure encodes the relevant information on the shape and smoothness properties of the random curves. It provides a way to optimally extract the intrinsically infinite variation carried by the random functions to lower dimension. In case the functional time series is weakly stationary, the second order dependence structure can be specified in the time domain through an infinite sequence of lag $h$ covariance operators
\[
\mathcal{C}_{h}= \mean\big[(X_0-m) \otimes (X_h-m) \big], \qquad h \in \znum,
\]
where $m$ is the mean function of $X$, which is the unique element of $H$ such that 
\[
\langle m, g \rangle =\E\langle X, g \rangle, \qquad g \in H.
\]  
Unlike independent functional data, where one only needs to consider the within curves dynamics as captured by the operator $\mathcal{C}_0$, functional time series require to take into account also \emph{all} the between curve dynamics as given by the infinite sequence of lag covariance operators $\mathcal{C}_h$ for $ h\ne 0$. The full second order dynamics are then more straightforwardly captured in the frequency domain, and an initial framework for Fourier analysis of random functions was therefore developed in \citet{PanarTav2013a}. Their framework of frequency domain-based inference is however restricted to processes for which the notion of a spectral density operator, defined as the Fourier transform of the sequence of $h$-lag covariance operators,
\[
\F_{\omega} =\frac{1}{2\pi} \sum_{h \in \znum} \mathcal{C}_h e^{-\im h \omega}, \qquad \omega \in [-\pi,\pi], 
\] 
exists. 
In this case, the autocovariance operator at lag $h$ can itself be represented as
\begin{align} \label{covInvFT1}
\mathcal{C}_h = \int_{-\pi}^{\pi} e^{\im h \omega} \mathcal{F}_{\omega} d\omega,
\end{align}
where the convergence holds in the appropriate operator norm. For processes of which $\F_{\omega}$ has absolutely summable eigenvalues,  \citet{Panar2013b} derived a functional Cram{\'e}r representation and showed that the eigenfunctions of $\F_{\omega}$ allow a harmonic principal component analysis, providing an optimal representation of the time series in finite dimension. \textcolor{black}{This was later relaxed in Tavakoli (2014) to processes with only weak spectral density operators implicitly defined by \eqref{covInvFT1}.} An optimal finite dimensional representation of a functional time series was also derived by \cite{Hormann2015} for $L^2_m$-approximable sequences under slightly different assumptions. In both works, the spectral density operator can be seen to take the same role as the covariance operator for independent functional data in the classical Karhunen--Lo{\`e}ve expansion \citep{Karhunen1947,Loeve1948}. \Citet{vde16}~extended frequency domain-based inference for functional data to locally stationary processes allowing thus to relax the notion of weak stationarity and to consider time-dependent second order dynamics through a time-varying spectral density operator. Since frequency domain-based inference does not require structural modeling assumptions other than weak dependence conditions, it has proved helpful in the construction of stationarity tests \citep[see e.g.,][]{avd16,vDCD18}\textcolor{black}{, but also in a variety of other inference problems \citep[see e.g.,][]{PhPan2018,LPapSap2018,HorKokNis17,vDD19}}. 

The aforementioned literature is restricted to processes for which the spectral density operators exist at all frequencies as elements of the space of trace class operators, $S_1(H)$. \textcolor{black}{This excludes many interesting processes for which the spectral density operator is not well defined at all frequencies or that have a spectral measure with discontinuities. For instance, processes with long memory caused by highly persistent cyclical or seasonal components arise quite naturally in a variety of fields such as hydrology or economics \citep[e.g.][]{McPol14}. A particular example of functional data are the supply and demand curves for electricity prices, which usually show a strong daily as well as weakly pattern \citep[e.g.][]{ZS16}.}
Since statistical inference techniques for this type of data must also take into account their within- and between curves dynamics, it is of importance to be able to develop frequency domain analysis under weaker conditions and to investigate under what conditions such an analysis is possible. In this note, we aim to provide the main building blocks for this relaxation and establish functional versions of the two fundamental results that lie at the core of frequency domain analysis for classical stationary time series: Herglotz's Theorem and the Cram{\'e}r Representation Theorem. It is worth remarking that we establish these two results under necessary conditions.  These results allow in particular to develop optimal finite dimension reduction techniques for such highly relevant applications, which are currently not available.  

The structure of this note is as follows. In section \ref{sec2}, we start by introducing the necessary notation and terminology. In section \ref{Herglotz}, we establish the existence of a functional Herglotz's theorem. For this, we make precise the concept of an operator-valued measure and the notion of operator-valued kernel functions. In section \ref{Gcram}, Herglotz's theorem is used to prove a generalized functional Cram{\'e}r representation for a large class of weakly stationary Hilbert-valued time series, including those with discontinuities in the spectral measure and long-memory processes. Finally, 
a Karhunen-Lo{\`e}ve expansion on the frequency components in the Cram{\'e}r representation is applied in order to obtain a harmonic principal component analysis of the series.

\section{Notation and preliminaries}\label{sec2}

\subsection{The function space}

We first introduce some necessary notation. Let $(T,\mathcal{B})$ be a measurable space with $\sigma$-finite measure $\mu$. Furthermore, let $E$ be a Banach space with norm $\norm{\cdot}_E$ and equipped with the Borel $\sigma$-algebra. We then define $L^p_E(T,\mu)$ as the Banach space of all strongly measurable functions $f:T\to E$ with finite norm
\[
\norm{f}_{L^p_E(T,\mu)}=\Big(\int \norm{f(\tau)}^p_E\,d\mu(\tau)\Big)^{\tfrac{1}{p}}
\]
for $1\leq p<\infty$ and with finite norm
\[
\norm{f}_{L^\infty_E(T,\mu)}=\inf_{\mu(N)=0}\sup_{\tau\in T\without N}\norm{f(\tau)}_E
\]
for $p=\infty$. We note that two functions $f$ and $g$ are equal in $L^p$, denoted as $f \overset{L^p}{=} g$, if $\norm{f-g}_{L^p_E(T,\mu)}=0$. If $E$ is a Hilbert space with inner product $\innerprod{\cdot}{\cdot}_E$ then $L^2_E(T,\mu)$ is also a Hilbert space with inner product
\[
\innerprod{f}{g}_{L^2_E(T,\mu)}=\int\innerprod{f(\tau)}{g(\tau)}_E\,d\mu(\tau).
\]
\textcolor{black}{For elements $f$ and $g$ of a Hilbert space $H$, we denote the inner product by $\innerprod{f}{g}$ and the induced norm by $\|f\|$.}
 
We shall extensively make use of linear operators on a Hilbert space $H$. A linear operator on a Hilbert space $H$ is a function $A:H\to H$ that preserves the operations of scalar multiplication and addition. We shall denote the class of bounded linear operators by $\Bol$ and its norm by $\snorm{\cdot}_{\lo}$. Furthermore, the class of trace class operators and Hilbert-Schmidt operators will be denoted by $S_1(H)$ and $S_2(H)$, respectively and their norms by $\snorm{\cdot}_{1}$ and $\snorm{\cdot}_{2}$. Equipped with these norms $(\lo(H), \snorm{\cdot}_\lo)$ and $(S_1(H), \snorm{\cdot}_1)$ form Banach spaces while $(S_2(H), \snorm{\cdot}_2)$ forms a Hilbert space with inner product $\innerprod{\cdot}{\cdot}_{S_2}$.   An operator $A \in \Bol$ is called self-adjoint if $\innerprod{Af}{g} = \innerprod{f}{A g}$ for all $f,g \in H$, while we say it is non-negative definite if $\innerprod{Ag}{g} \ge 0$ for all $g \in H$. It will be convenient to denote the respective operator subspaces of self-adjoint and non-negative operators with ${(\cdot)}^\dagger$ and ${(\cdot)}^+$, respectively. It is straightforward to verify that  $\lo(H)^+\subseteq \lo(H)^\dagger$ and $S_p(H)^+\subseteq S_p(H)^\dagger$. \textcolor{black}{Finally, we denote $O_H$ for the zero operator on $H$ and denote the topological dual space of a Banach space $B$ by $(B)^\prime$ and similar for appropriate subspaces thereof.}

\subsection{Functional time series}\label{FTS}

We define a functional time series $X=\{X_t \colon t \in \znum\}$ as a sequence of random elements on a probability space $(\Omega,\mathcal{A},\prob)$ taking values in a separable Hilbert space $H$ \textcolor{black}{such as, for instance, the space $L^2([0,1])$ of all square integrable functions on the interval $[0,1]$. Throughout this text, we consider functional time series that are  weakly stationary in the usual sense, that is, the first and second moments exist and are invariant under translation in time. More precisely, $X$ is weakly stationary if $\mean\norm{X_t}^2<\infty$ for all $t\in\znum$, $X$ has constant mean functions $\mean(X_t)=m$ for all $t\in\znum$, and the second moment tensors of $X$ satisfy $\mean(X_t\otimes X_s)=\mean(X_{t-s}\otimes X_0)$ for all $t,s\in\znum$. We note that in this case the random elements $X_t$ belong to the Hilbert space $\mathbb{H}=L^2_{H}(\Omega,\prob)$ of all $H$-valued random variables $X$ with $\mean\norm{X}^2<\infty$.
The inner product on this space $\mathbb{H}$ is denoted by $\innerprod{\cdot}{\cdot}_{\mathbb{H}}= \E\innerprod{\cdot}{\cdot}$ and the induced norm is denoted by $\|\cdot\|_{\mathbb{H}}$.}
\textcolor{black}{Without loss of generality, we assume that the mean function $m$ is zero. In this case, the $h$-th lag covariance operators $\mathcal{C}_{h}$ can be defined by
\[\mathcal{C}_h = \E (X_h \otimes X_0).\]
The condition $\mean\|X_0\|^2<\infty$ ensures that the covariance operators $\mathcal{C}_h$ for $h \in \znum$ belong to $S_1(H)$. We elaborate on this in Section \ref{Gcram}. 
%
}

\textcolor{black}{The available literature on frequency domain analysis for weakly stationary functional time series has focused on so-called `short-memory' processes, that is, processes of which the dependence structure decays at a sufficiently fast rate, namely }
\begin{align} \label{decayCh}
\sum_{h \in \znum} \snorm{\mathcal{C}_h}_1 < \infty.
\end{align}
Under this condition, it is straightforward to show that the autocovariance operator $\mathcal{C}_h$ forms a Fourier pair with the spectral density operator given by
\[
\F_{\omega} =\frac{1}{2\pi} \sum_{h \in \znum} \mathcal{C}_h e^{-\im h \omega}, 
\] 
where the convergence holds in $\snorm{\cdot}_1$ and the spectral density operator acts on $H$. Under condition \eqref{decayCh},  a classical Ces{\`a}ro-sum argument \citep[see e.g.][]{b81} was used by \citet{PanarTav2013a} to derive that $\F_{\omega}$ is non-negative definite and hence belongs to $S_{1}(H)^+$. As already mentioned in the introduction, the autocovariance operator at lag $h$ itself can then be represented as
\begin{align} \label{covInvFT}
\mathcal{C}_h = \int_{-\pi}^{\pi} e^{\im h \omega} \mathcal{F}_{\omega} d\omega,
\end{align}
where the convergence holds in $\snorm{\cdot}_1$. \citet{Panar2013b} showed that a zero mean \textcolor{black}{weakly} stationary functional time series $X$ satisfying condition \eqref{decayCh} admits a functional spectral representation of the form
\begin{align}\label{eq:cramerwkp}
X_t = \int_{-\pi}^{\pi} e^{\im\omega t}\,dZ_{\omega} \qquad \text{a.s.}, 
\end{align}
where $Z_{\omega}$ is a functional orthogonal increment process such that, for fixed $\omega$, $Z_{\omega}$ is a random element in $H$ with $\E \| Z_{\omega}\|_2^2 = \int_{-\pi}^{\omega}\snorm{ \mathcal{F}_{\lambda}}_1 d \lambda$. \textcolor{black}{If the summability conditions as in \eqref{decayCh} do not hold, the spectral density operators might not necessarily exist as elements of $S^{+}_1(H)$ and, as a consequence, \eqref{covInvFT} and \eqref{eq:cramerwkp} might no longer hold true. Tavakoli (2014) studied the frequency domain representations for processes that violate the summability condition \eqref{decayCh} but possess a weak spectral density operator $\mathcal{F}_\omega$ implicitly defined as an element of $L^p([-\pi,\pi],S_1(H))$ with $1< p\leq\infty$ satisfying \eqref{covInvFT}. In the next section, we provide a frequency domain representation for general non-negative definite autocovariance operator functions $\{\mathcal{C}_h\}$. In particular, we do not require any assumptions on the rate of decay or the existence of a (weak) spectral density operator. In section 4, we then use this representation to derive a functional Cram\'{e}r representation for general weakly stationary functional time series, which can be seen as a true generalization of the classical Cram\'{e}r representation theorem to functional time series.}

\section{Herglotz's Theorem on a function space}\label{Herglotz}

In this section we derive a functional generalization of the classical Herglotz's theorem. Note that it is intuitive from \eqref{covInvFT} that if we do not have a spectral density operator, then the measure itself must be operator-valued for the equation to be balanced. This raises the following questions: firstly, does an operator of the form
\begin{align}
\int_{-\pi}^{\pi} e^{\im h \omega}\,d\Fm(\omega), \qquad h \in \znum,
\end{align}
where $\Fm$ is an operator-valued measure on $[-\pi,\pi]$, exist? Secondly, what properties must an operator possess to be represented by such an integral?  From the classical Herglotz's theorem, we know that the non-negative definite complex-valued function on the integers are precisely those that can be identified to have a frequency domain representation with respect to a finite Radon measure.

\begin{theorem}[Herglotz's theorem] \label{thm:Herglotzscalar}
A function $\gamma(\cdot): \znum \to \cnum$ is non-negative definite if and only if
\[
\gamma(h) =\int_{\pi}^{\pi} e^{\im h \omega}\,dF(\omega), \qquad h  \in \znum
\]
where $F(\cdot)$ is a right-continuous, non-decreasing bounded function on $[-\pi,\pi]$ with $F(-\pi) = 0$. 
\end{theorem}

Here, the so-called spectral distribution function $F(\cdot)$ with $F(-\pi)=0$ is uniquely determined by the covariance function $\gamma(h)$, $h \in \znum$. To extend this result to our functional setting, we require the notion of non-negative definiteness of operator-valued functions on $\znum$ as well as definitions of operator-valued measures and of integrals with respect to such measures. For the former, we proceed as in the scalar-valued case with the help of non-negative operator-valued kernels.

\begin{definition} \label{pdkernel}\mbox{}
\begin{romanlist}
\item
A function $c: \znum \times \znum \to \Bol$ is called a non-negative definite $\Bol$-valued kernel if 
\[
\sum_{i,j=1}^{n} \innerprod{c(i,j)\,g_j}{g_i} \ge 0
\]
for all $g_1,\ldots,g_n \in H$, and $n \in \nnum$. 
\item
A function $\mathcal{C}:\mathcal\znum\to\Bol$ is called non-negative definite if the kernel $c: \mathcal \znum\times\znum\to \Bol$ defined by $c(i,j) = \mathcal{C}(i-j)$ is a non-negative definite kernel.
\end{romanlist}
\end{definition}

Non-negative definite operator-valued kernels are an extremely powerful tool in functional analysis, especially in operator theory and representation theory. A generalization of this concept to general involutive semigroups other than $\znum$ can be found in  \citet{Neebbook}. For functional time series, Definition \ref{pdkernel} provides a link between the properties of the covariance kernel and the corresponding covariance operator viewed as a function on the integers. More specifically, we have the following result.

\begin{proposition} \label{prop:nonnegCh}
Let $\{X_t \colon t\in \znum\}$ be a zero-mean weakly stationary functional time series with covariance operators $\mathcal{C}_{h}$, $h \in \znum$. Then $\mathcal{C}_{(\cdot)}: \znum \to \Bol$ is a non-negative definite function. 
\end{proposition}
\begin{proof}
The operator-valued kernel $c(i,j) =\mathcal{C}_{i-j}$, $i,j \in \znum$ satisfies 
\[
\sum_{i,j=1}^{n} \innerprod{\mathcal{C}_{(i-j)}g_j}{g_i} = \sum_{i,j=1}^{n} \innerprod{\E (X_{i} \otimes X_{j})g_j}{g_i} =\E \bignorm{\sum_{i=1}^{n} \innerprod{X_{i}}{ g_i}}^2 \ge 0
\]
and is therefore non-negative definite by Definition \ref{pdkernel}.
\end{proof}


\subsection{\textcolor{black}{$\lo(H)$-valued measures}}

In this subsection, we explain how $\lo(H)$-valued measures can be defined. Let us first provide some intuition by making an analogy with positive scalar-valued measures. Recall that a positive measure $\mu$ on a measurable space $(T,\mathcal{B})$ is defined as a countably additive map on $\mathcal{B}$ taking values in the compactification $[0,\infty]$ of the set $\rnum^+:=[0,\infty)$. The compactification is necessary for $\sigma$-additivity to hold (although for finite measures the compact subset $[0,\mu(T)]$ is sufficient). We note that $\rnum^+$ is an example of a pointed convex cone, that is, it is a convex nonempty subset of $\rnum$ that is closed under non-negative scalar multiplication and contains the zero element. It is moreover dense in $\rnum^+_{\infty}: =[0,\infty]$. Taking this view, a positive measure more generally can be defined as a countably additive map taking values in the compactification of a pointed convex cone. For the purpose of this paper, we are solely interested in measures taking values in the  compactification of  $\lo(H)^+$, the pointed convex cone of $\lo(H)$ consisting of all non-negative elements $\lo(H)$. \textcolor{black}{We now explain heuristically how such a measure can be defined. The technical argument and more details on properties of cones are relegated to Appendix \ref{sec:proofssec3}.} For the general theory on cone-valued measures we refer to \citet{Neebbook} and \citet{Glockner2003}.\\

\textcolor{black}{Essential in the construction of the measure is the following duality between the underlying real Banach space of self-adjoint elements $\lo(H)^\dagger$ and the real Banach space of self-adjoint trace class operators, $S_{1}(H)^{\dagger}$. More specifically, by the duality pairing
\[
\innerprod{\cdot}{\cdot} : \Bol^\dagger \times S_{1}(H)^\dagger \to \rnum, \qquad\innerprod{B}{A} = \tr(B A)\tageq \label{eq:pairapp}
\]
and defining $\phi_B:S_{1}(H)^\dagger \to \rnum$ with $\phi_B(A)=\innerprod{B}{A}=\tr(B A)$, we can identify $(\lo(H)^\dagger, \opnorm{\cdot})$ with the topological dual space of $(S_{1}(H)^{\dagger}, \snorm{\cdot}_{1})$ through the isometric isomorphism
\[
\phi:\lo^\dagger(H) \to (S_1(H)^\dagger)^\prime,\quad B\mapsto \phi_B, 
\]
that is, we have $\lo(H)^\dagger \cong (S_1(H)^\dagger)^\prime$. This isomorphism provides us with a natural notion of convergence for a sequence of operators in $\lo(H)^\dagger$, known as the ultraweak topology. 
\begin{definition}\label{def:uw}
\textit{A sequence of operators $\{B_n\} \in \lo(H)^\dagger$ converges in the ultraweak topology to ${B_0}$ if
$\phi_{B_n}(A) \to \phi_{B_0}(A)$ for all $A \in S_1(H)^\dagger$ as $n\to\infty$.}
\end{definition}
Thus, the ultraweak topology is the coarsest topology such that the pointwise evaluations $\phi_B(A)$ for all $A\in S_1(H)^\dagger$ are continuous as functions in $B$.\\
The isomorphism $\phi$ when restricted to the cone $\lo(H)^+$ of non-negative definite bounded operators suggests a similar result between $\lo(H)^+$ and $S_1(H)^+$, the non-negative definite trace class operators, namely
$\lo(H)^+ \cong (S_1(H)^+)^\prime$. However, a slight problem arises from the fact that $\lo(H)^+$, $S_1^+(H)$, and $\rnum^+$ are not vector spaces in the strict sense because they are not closed under negative scalar multiplication.  Nevertheless, they can be regarded as topological monoids with respect to addition as they are closed under addition, have a zero element, and the addition is continuous (see Appendix \ref{ap:backgr}). Similarly, the restricted mappings $\phi_B:S_1(H)^+\to\rnum^+$ for $B\in\lo(H)^+$ now become monoid homomorphisms, that is, they preserve addition and the zero element. Moreover, it can be shown that the mapping $\phi$ in this case provides an isomorphism (Appendix \ref{sec:apem}) between the monoid $\lo(H)^+$ and the monoid $\operatorname{Hom}_{mon}(S_1(H)^+,\rnum^+)$, which consists of \textit{all} monoid homomorphisms from $S_1(H)^+$ to $\rnum^+$. That is, we have
\[
\phi(\lo(H)^+)=\operatorname{Hom}_{mon}(S_1(H)^+,\rnum^+). \tageq \label{eq:hommon}
\]
Since the isometry property of $\phi$ is preserved when restricted to the cone $\lo(H)^+$ we can obtain (Theorem \ref{prop:fact2}) an isometric isomorphism 
\[\phi(\lo(H)^+)= \operatorname{Hom}_{mon}(S_1(H)^+,\rnum^+) \cong_{m} \lo(H)^+,\] where the subscript $m$ in $\cong_m$ emphasizes that the isomorphism is between monoids. Since the set \eqref{eq:hommon} is dense in the compact set $\text{Hom}_{mon}(S_{1}(H)^+, \rnum^+_{\infty})$ \citep{Glockner2003}, the compactification of $\lo(H)^+$ naturally can be defined by
\[ \tageq \label{eq:comp}
\Bcom= \text{Hom}_{mon}(S_1(H)^+, \rnum^+_{\infty}).
\]
The set $\text{Hom}_{mon}(S_1(H)^+, \rnum^+_{\infty})$ inherits the notion of convergence as given in Definition \ref{def:uw}, ensuring that addition is continuous with respect to this notion. The argument is technical and can be found in Appendix \ref{sec:proofssec3}. With this, $\lo(H)^+$-valued measures can now be defined as follows.
\begin{definition} \label{def:meas}
\textit{Let $(T,\mathcal{B})$ be a measurable space.  A mapping $\mu: \mathcal{B} \to \Bcom$ is an $\lo(H)^+$-valued measure on $(T,\mathcal{B})$ if it is countably additive and $\mu(\emptyset)=O_H$.  The measure $\mu$ is called finite if $\mu(T) \in \text{Hom}_{mon}(S_1(H)^+, \rnum^+) \cong_m \lo(H)^+$. It is called $\sigma$-finite if $E$ is the countable union of measurable sets with finite measure, that is, if $E = \bigcup_{i=1}^{\infty} E_i$  and where $\mu(E_i) \in \lo(H)^+$ for all $i \in \nnum$.}
\end{definition}
Observe that by \eqref{eq:comp}, countably additivity in $\Bcom$ holds with respect to the ultraweak topology, that is, for any sequence $\{E_i\}_{i=1}^\infty$ of pairwise disjoint sets in $\mathcal{B}$, we have
\[
\mu(\bigcup_{i=1}^{\infty} E_i)(A) = \sum_{i=1}^{\infty} \mu(E_i)(A) 
\]
for all $A\in S_1(H)^+$, where we note that $\mu(E_i)(A) =\tr(\mu(E_i)A)$. This implies also that a positive $\lo(H)^+$-valued measure $\mu$ can be represented by a family of positive scalar-valued measures $\{\mu_A\}_{A \in S_1(H)^+}$ given by the evaluation functions
\[
\mu_A(E) = \mu(E)(A) 
\]
for all $E\in \mathcal{B}$. Conversely, a family of positive scalar-valued measures $\{\mu_A\}_{A \in S_1(H)^+}$ defines a positive $\lo(H)^+$-valued measure $\mu$ if and only if for all $E\in \mathcal{B}$ the mapping $A\mapsto\mu_A(E)$ is a monoid homomorphism. We summarize this important relation between the operator-valued measure and the corresponding family of scalar-valued measures in the following theorem, which is a simplification of Theorem I.10 of \citet{Neeb98}.}
\begin{theorem} \label{thm:Cons_m}
Let $\{\mu_A\}_{A \in S_1(H)^+}$ be a family of non-negative measures on the measurable space $(T, \mathcal{B})$ s.t. for each borel set $E \subseteq V$ the assignment $A \mapsto \mu_A(E)$ is a monoid homomorphism. Then there exists for each $E \in \mathcal{B}$ a unique element $\mu(E) \in \Bcom$ with $\mu(E)(A) = \mu_A(E)$ for all $A \in  S_1(H)^+$ and the function $\mu : \mathcal{B} \to \Bcom$ is a $\lo(H)^+$-valued measure.
\end{theorem}
In particular, it follows that a $\Bcom$-valued Radon measure $\mu$ on a locally compact space $T$ is a measure with the property that, for each non-negative trace class operator $A \in S_1(H)^+$, the measure $\mu_A: \mathcal{B} \to \rnum^+_{\infty}$ is a finite non-negative Radon measure on $T$. Moreover, integrability of a function with respect to the scalar measure furthermore implies integrability of the function with respect to the $\lo(H)^+$-valued measure. 

\subsection{Functional Herglotz's theorem}

We are now ready to prove the following generalization of Theorem \ref{thm:Herglotzscalar}. 
\begin{theorem}[Functional Herglotz's Theorem]\label{Herglotzthm}

A function $\Gamma:\znum \to \Bol$ is non-negative definite if and only if
\begin{align} \label{eq:Herglotzint}
\Gamma(h) = \int_{-\pi}^{\pi} e^{\im h \omega} d\Fm(\omega) \qquad h \in \znum,
\end{align}
\textcolor{black}{where $\Fm$ is a finite $\Bcom$-valued measure on $[-\pi,\pi]$ with $\Fm(-\pi)=O_H$. The finite measure $\Fm$ is uniquely determined by $\Gamma(h), h \in \znum$.}
\end{theorem}

\begin{proof}
 \textcolor{black}{Suppose first that $\Gamma(h)$ admits the representation \eqref{eq:Herglotzint} with respect to a finite $\Bcom$-valued measure $\Fm$ on $[-\pi,\pi]$.  Note first that, since $\Fm$ is a finite measure, the integral is well-defined. By (ii) of Definition \ref{pdkernel}, it is sufficient to show that the operator-valued kernel $\gamma(h_1,h_2) : =\Gamma(h_1-h_2)$ is non-negative definite. Using (i) of Definition \ref{pdkernel} and that $\Fm$ \textcolor{black}{takes values in $\lo(H)^+$}, we have for all  $h_1,\ldots, h_n \in \znum$ and $g_1,\ldots,g_n \in H$. 
\begin{align*}
 \sum_{j,k=1}^n \innerprod{\gamma(h_j,h_k)g_k}{g_j} &= \sum_{j,k=1}^n \innerprod{ \int_{[-\pi,\pi]} e^{\im (h_j-h_k) \omega} d\Fm(\omega)g_k}{g_j} 
 \\& = \int_{[-\pi,\pi]}\innerprod{\Fm(d\omega) \lsum_{k=1}^n e^{-\im h_k\omega}g_k} {\lsum_{j=1}^n e^{-\im h_j \omega}g_j} \textcolor{black}{\geq 0}.
\end{align*}
} Conversely, suppose that $\Gamma(\cdot)$ is a $\lo(H)^+$-valued function on $\znum$ and let $A$ be a an element of $S_1(H)^+$. Define the function $\Gamma_A: \znum \to \cnum$ by $\Gamma_A(h) = \Tr(\Gamma(h)A)$. Since the square root of both $\Gamma(h)$ and $A$ are well-defined and the trace satisfies $\Tr(\Gamma(h)A)=\Tr(A\Gamma(h))$, it is direct that $\Gamma_A$ is a non-negative definite function on the integers.  By the classical Herglotz's theorem (Theorem \ref{thm:Herglotzscalar}), we therefore have the representation
\[
\Gamma_A(h) = \int_{-\pi}^{\pi} e^{\im h \omega} d\Fm_A(\omega) \qquad h \in \znum,\tageq \label{eq:HergGamA}
\]
where $\Fm_A(\cdot)$ with $\Fm_A(-\pi)=0$ is a uniquely determined Radon measure on $[-\pi,\pi]$. More specifically, the scalar Herglotz theorem implies that $\Fm_A$ is a right-continuous non-decreasing bounded function on $[-\pi,\pi]$. The measure is finite since $C_A(\znum) = C_A(0) = \Fm_A([-\pi,\pi]) < \infty$. \textcolor{black}{Let then $A_1, A_2 \in S_1(H)^+$ and observe that $A_1+A_2 \in S_1(H)^+$ and thus $\Gamma_{A_1+A_2}(h)$ is again a non-negative definite function on $\mathbb{Z}$. Hence, another application of the classical Herglotz theorem to $\Gamma_{A_1+A_2}(\cdot)$ yields
\[
 \Gamma_{A_1+A_2}(h) = \int_{-\pi}^{\pi} e^{\im h \omega} d\Fm_{A_1+A_2}(\omega) \qquad h \in \znum,
\]
whereas by \eqref{eq:HergGamA} $\Gamma_{A_1}(h)+\Gamma_{A_2}(h)= \int_{-\pi}^{\pi} e^{\im h \omega} d\Fm_{A_1}(\omega) + \int_{-\pi}^{\pi} e^{\im h \omega} d\Fm_{A_2}(\omega)$ and since linearity of the trace yields that $\Gamma_{A_1}(h)+\Gamma_{A_2}(h) = \Gamma_{A_1+A_2}(h)$, we obtain that $\Fm_{A_1}(\cdot)+\Fm_{A_2}(\cdot) = \Fm_{A_1+A_2}(\cdot)$ for any $A_1, A_2 \in S_1(H)^+$. This demonstrates that $A \mapsto \Fm_A(E)$ is a monoid homomorphism $S_1(H)^{+} \to \rnum^+$ with respect to addition for each $E \in \mathcal{B}$. By Theorem \ref{thm:Cons_m}, the family of measures $\{\Fm_A\}_{A \in S_1(H)^{+}}$ therefore uniquely identifies an element $\Fm(E) \in \Bcom$ on the measurable space $([-\pi,\pi],\mathcal{B})$ with $\Fm(E)(A) =\Tr(\Fm(E) A) = \Fm_A(E)$ for all $A \in S_1(H)^+$. Since $\Fm_A([-\pi,\pi]) =\tr(\Fm([-\pi,\pi])A))<\infty,\, \forall A \in S_1(H)^+$, we obtain that $\Fm$ is a finite operator-valued measure on $[-\pi,\pi]$, i.e, $\Fm([-\pi,\pi]) \in \lo(H)^+$. Finally, note from the above that
\[
\tr(\Gamma(h)A) = \tr\Big(\int_{-\pi}^{\pi} e^{\im h \omega} d \Fm(\omega)  A\Big) \quad \forall A \in S_1(H)^+
\]
and 
thus $\Gamma(h) \,-\int_{-\pi}^{\pi} e^{\im h \omega} d \Fm(\omega) =O_H$ for all $h \in \znum$ which proves that $\Fm$ is uniquely determined by $\Gamma(h), h \in \znum$.} \end{proof}
\textcolor{black}{
\begin{remark}[Analogy to the classical Herglotz theorem]
Theorem \ref{Herglotzthm} tells us that $\Fm$ is a finite $\Bcom$-valued measure on $[-\pi,\pi]$, i.e., $\Fm([-\pi,\pi]) \in \lo(H)^+$. To make the analogy to the classical Herglotz theorem,we note that $\Fm( [-\pi,\omega])$ can be seen as a non-decreasing, right-continuous operator-valued function in $\omega \in [-\pi,\pi]$. In particular, we identified the measure with a family of finite scalar-valued measures $\{\Fm_A\}_{A \in S_1(H)^{+}}$ on $[-\pi,\pi]$, which satisfied the conditions of the classical Herglotz theorem. Consequently, $\Fm$ is right-continuous in the sense that $\lim_{\omega \downarrow \omega_o} \Fm([-\pi,\omega]) =\tr(\Fm(-\pi,\omega_o])$ where the convergence holds in the ultraweak topology. Moreover, by definition of it being a $\lo(H)^+$-valued measure on $[-\pi,\pi]$, it is non-decreasing in $[-\pi,\pi]$, i.e.,  for all $\omega_2 \ge \omega_1$, $\omega_1, \omega_2 \in [-\pi,\pi]$ we have $\Fm([\omega_1,\omega_2]) = \Fm(\omega_2)-\Fm(\omega_1) \ge O_H$.
\end{remark}}


\section{A generalized functional Cram{\'e}r representation}\label{Gcram}

The Spectral Representation Theorem \citep[][]{Cramer1942}, often called the {\em Cram{\'e}r re\-presentation}, is as fundamental to frequency domain analysis as Wold's representation is to the time domain. It asserts that every (finite dimensional) zero-mean weakly stationary process can be represented as a superposition of sinusoids with random amplitudes and phases that are uncorrelated. An important ingredient in establishing this classical theorem is the existence of an isometric isomorphism that allows to identify a weakly stationary time series on the integers with an orthogonal increment process on $[-\pi,\pi]$. 
As already mentioned, an initial generalization of the Cram{\'e}r representation to \textcolor{black}{weakly stationary functional time series} was first considered by \citet{Panar2013b}, but is restricted to processes for which the assumption $\sum_{h \in \znum} \snorm{\mathcal{C}_h}_1 < \infty$ holds. In this section, we shall use the established functional Herglotz's theorem (Theorem \ref{Herglotz}) to derive a functional Cram{\'e}r representation that can be seen as a true generalization of the classical theorem to the function space. In addition, we establish \textcolor{black}{a Cram{\'e}r--Karhunen--Lo{\`e}ve representation --a term first coined by \citet{Panar2013b}--,} and a harmonic prinicipal component analysis for a very general class of processes of which the spectral measure can have finitely many discontinuities.

\textcolor{black}{We first show that for a weakly stationary functional time series the full second order structure is given by a sequence of trace class operators.}
\begin{corollary} \label{cor:Ch}
\textcolor{black}{Let $X$ be a weakly stationary functional time series. Then} the sequence of lag covariance operators $\{\mathcal{C}_h\}_{h \in \znum}$ belongs to $S_1(H)$.
\end{corollary}
\begin{proof}
\textcolor{black}{By Jensen's inequality, we have $\snorm{\E(X_h \otimes X_0)}_1 \le \E\snorm{X_h \otimes X_0}_1$.
$X_h \otimes X_0$ is therefore a random element of $S_1(H)$ if $\E\sum_i |\innerprod{(X_h \otimes X_0)e_i}{e_i}|<\infty$. By the Cauchy schwarz inequality and Parseval's identity
\begin{align*}
 \E \sum_i |\innerprod{(X_h \otimes X_0)e_i}{e_i}| & 
\le \E \sqrt{\sum_i |\innerprod{e_i}{X_0}|^2} \sqrt{\sum_i |\innerprod{X_h}{e_i}|^2} 
\\& \le   \|X_h\|_{\mathbb{H}} \|X_0\|_{\mathbb{H}}= \|X_0\|^2_{\mathbb{H}} <\infty,
\end{align*}
where the last inequality follows again from the Cauchy Schwarz inequality and the equality follows from weak stationarity. }
\end{proof}

\textcolor{black}{As in the classical case, the proof of the functional Cram{\'e}r representation is based on showing that the mapping}
\[ X_t \mapsto e^{\im t \cdot } \]
forms an Hilbert space isometric isomorphism between 
$L^2_H(\Omega, \prob)$ and 
$L^2([-\pi,\pi], \mu_{\Fm})$ where we define 
\[\mu_{\Fm}(E) = \snorm{\Fm(E)}_1 \tageq \label{mufm}\] for all Borel sets $E \subseteq [-\pi,\pi]$. Here, $\Fm$ is the operator-valued measure on $[-\pi,\pi]$ induced by the sequence of covariance operators $\{C_h\}_{h \in \znum}$ of $\{X_t \colon t \in \znum \}$. Before we derive the properties of the mapping, we have to verify that this indeed defines a measure. This is the contents of the following lemma. 

\begin{lemma} \label{prop:muFm}
\textcolor{black}{Let $X$ be a weakly stationary functional time series.} Then the function $\mu_{\Fm}$ defined in \eqref{mufm} is a finite scalar-valued measure on $[-\pi,\pi]$.    
\end{lemma}
\begin{proof}[Proof of Lemma \ref{prop:muFm}]
By Proposition \ref{prop:nonnegCh}, the covariance function $\mathcal{C}_{(\cdot)}: \znum \to S_1(H)$ of a weakly stationary $H$-valued time series is non-negative definite. Using Corollary \ref{cor:Ch}, Theorem \ref{Herglotzthm} implies this function uniquely determines a $S_1(H)^+$-valued measure $\Fm$ on $[-\pi,\pi]$. Using the properties of $\Fm$, it is now straightforward to verify that the function $\mu_{\Fm} : \mathcal{B} \to [0,\infty]$ is a non-negative scalar-valued measure on the measurable space $([-\pi,\pi],\mathcal{B})$. Firstly, for each Borel set $E \subseteq [-\pi,\pi]$ and every $e \in H$, we have $\innerprod{\Fm(E) e}{e}\ge 0$ and thus  $\mu_{\Fm} = \Tr(\Fm) \ge 0$. Secondly, $\innerprod{\Fm(\emptyset) e}{e} = \innerprod{O_H e}{e} =\mu_{\Fm}(\emptyset) = 0$ which follows by definition of $\Fm$. Thirdly, for all countable selections of pairwise disjoint set $\{E_i\}_{i \in \nnum}$ in $\mathcal{B}$, countable additivity of $\Fm$ yields
\[ \mu_{\Fm}\big(\bigcup_{i=1}^{\infty}E_i\big) = \sum_{j}^{\infty}\innerprod{\Fm( \bigcup_{i=1}^{\infty}E_i) e_j}{e_j} = \sum_{j=1}^{\infty}\innerprod{\sum_{i}^{\infty}\Fm( E_i) e_j}{e_j},\]
where $\{e_j\}_{j \in \nnum}$ is an orthonormal basis of $H$. By continuity of the inner product and the fact that $\Fm([-\pi,\pi])< \infty$, Fubini's theorem implies
\[ \sum_{i=1}^{\infty} \sum_{j=1}^{\infty}\innerprod{\Fm( E_i) e_j}{e_j} =\sum_{i=1}^{\infty} \snorm {\Fm(E_i)}_1 = \sum_{i=1}^{\infty} \mu_{\Fm}(E_i)  \]
and thus $\mu_{\Fm}$ is countably additive. Finally, since $\Fm$ is $S_1(H)^+$-valued measure on $[-\pi,\pi]$ it is direct that $\mu_\Fm([-\pi,\pi]) <\infty$.
\end{proof}

Additionally, to be able to properly define the spectral representation we require the notion of a $H$-valued orthogonal increment process. 
\begin{definition}\label{Zproc}
A $H$-valued random process $\{Z_\omega \colon -\pi \le \omega \le \pi \}$ defined on a probability space $(\Omega, \mathcal{A},\prob)$ is a functional orthogonal increment process, if for all $g_1, g_2 \in H$ and $ -\pi \le \omega \le \pi$
\begin{romanlist}
\item the operator $\E(Z_\omega \otimes Z_{\omega})$ is an element of $S_{1}(H)^+$
\item  $\E\innerprod{Z_\omega}{g_1}= 0$
\item $\innerprod{\E\big((Z_{\omega_4} - Z_{\omega_3})\otimes {(Z_{\omega_2} - Z_{\omega_1})}\big)g_1}{g_2}= 0, \qquad (\omega_1,\omega_2] \cap (\omega_3,\omega_4]=\emptyset$
\item $\innerprod{\E\big((Z_{\omega+\varepsilon} - Z_{\omega})\otimes {Z_{\omega+\varepsilon} - Z_{\omega}}\big)g_1}{g_2}\to 0$ as $\varepsilon \downarrow 0$.
\end{romanlist}
\end{definition}
To establish the isomorphism, let $\mathcal{H} \subset \Hspace$ denote the space spanned by all finite linear combinations of the random functions $X_t$, i.e., $\mathcal{H}= {\text{sp}}\{X_t \colon t \in \znum\}$. We remark that the inner product on $\Hspace$ satisfies
\[
\innerprod{X_1}{X_2}_{\mathcal{\Hspace}} =\Tr(\E(X_1 \otimes X_2) ). \tageq \label{eq:inTr}
\]
Furthermore let the space $\mathscr{H}$ denote the space of all square-integrable functions on $[-\pi,\pi]$ with respect to the measure $\mu_\Fm$, i.e., $\mathscr{H}= L^2([-\pi,\pi], \mu_{\Fm})$. This space becomes a Hilbert space once we endow it with the inner product 
\[
\innerprod{f}{g}_{\mathscr{H}} =  \int_{-\pi}^{\pi} f(\omega) \widebar{g(\omega)}d \mu_\Fm(\omega) =\Tr(\int_{-\pi}^{\pi} f(\omega)  \overline{g(\omega)}d\Fm(\omega)) \qquad f,g \in H,
\]
where the last equality follows by non-negative definiteness of $\Fm$ and linearity of the trace operator. 

\begin{theorem} \label{isomorph}
\textcolor{black}{Let $X$ be a weakly stationary functional time series, and let} $\Fm$ be the $S_{1}(H)^+$-valued measure corresponding to the process $\{X_t\}$. Then there exists an isometric isomorphism $\mathcal{T}$ between $\overline{\text{sp}}\{X_t\}$ and $L^2([-\pi,\pi], \mu_{\Fm})$ such that 
\[\mathcal{T}X_t = e^{\im t \cdot}, \qquad t \in \znum.\]
The process defined by 
\[Z_{\omega} = \mathcal{T}^{-1}\big(1_{(-\pi,\omega]}(\cdot)\big)\] 
is then a functional orthogonal increment process of which the covariance structure is uniquely determined by $\Fm$ and satisfies
\[ 
\mathbb{E}\big[(Z_{\omega}-Z_{\lambda}) \otimes( Z_{\omega}-Z_{\lambda})\big]= \Fm(\omega)-  \Fm(\lambda), \quad -\pi \le \lambda \le \omega \le \pi. 
\]
\end{theorem}

\begin{proof}[Proof of Theorem \ref{isomorph}]
Consider first the mapping $\mathcal{T} : \mathcal{H} \to \mathscr{H}$ given by
\[
\mathcal{T} (\sum_{j=1}^{n}a_j X_{t_j})= \sum_{j=1}^{n}a_j e^{\im t \cdot}  \qquad t \in \znum.
\]
It is straightforward to see the mapping is linear and preserves inner products. Let $Y = \sum_{j=1}^{n}a_j X_{t_j}$ and $W = \sum_{j=1}^{n}b_j X_{t_j}$. By Theorem \ref{Herglotz},
\begin{align*}
\innerprod{\mathcal{T}Y}{\mathcal{T}W}_\mathscr{H} & =\sum_{i,j=1}^{n}a_i \overline{b}_j \innerprod{e^{\im \cdot t_i}}{e^{\im \cdot t_j}}_\mathscr{H} 
\\& =\sum_{i,j=1}^{n}a_i \overline{b}_j \int_{-\pi}^{\pi} e^{\im \lambda (t_i-t_j)} d\mu_{\Fm}(\lambda)
\\& =\sum_{i,j=1}^{n}a_i \overline{b}_j \Tr(\int_{-\pi}^{\pi} e^{\im \lambda (t_i-t_j)} d{\Fm}(\lambda))
\\& =\lsum_{i,j=1}^{n}a_i \overline{b}_j \Tr(C_{i-j} )= \lsum_{i,j=1}^{n}a_i \overline{b}_j \innerprod{X_{t_i}}{ X_{t_j}}_{\Hspace} = \innerprod{Y}{W}_\Hspace.
\end{align*}
For the extension of the isomorphism over the closure of $\mathcal{H}$ onto the closure of $\mathscr{H}$, note that if $Y$ is an element of $\bar{\mathcal{H}}$ then there must exist a sequence $\{Y_n\}_{n \ge 1} \in \mathcal{H}$ converging to $Y$. Denote  $\mathcal{T}({Y})$ to be the limit of $\mathcal{T}(Y_n)$, i.e., 
\[\mathcal{T}(Y) = \lim_{n \to \infty}\mathcal{T}(Y_n).\]
Since $\{Y_n\}$ is a Cauchy sequence and $\mathcal{T}$ norm-preserving, $\{\mathcal{T}Y_n\}$ is a Cauchy sequence in $L^2([-\pi,\pi],\mu_{\Fm})$ and thus $\mathcal{T}(Y) \in \bar{\mathscr{H}}$. If there is another sequence $\{Y^{'}_n\} \in \mathcal{H}$ converging to ${Y}$, then the limit must be unique since
\[
\lim_{n \to \infty}\norm{\mathcal{T}(Y_n)- \mathcal{T}(Y^{'}_n)}_{\mathscr{H}} =\lim_{n \to \infty} \norm{Y_n - Y^{'}_n}_{\Hspace} = 0,
\]
and therefore the extension is well-defined. Preservation of linearity and the isometry property are straightforward from linearity of $\mathcal{T}$ on $\mathcal{H}$ and continuity of the inner product, respectively. To show that the closure of $\mathscr{H}$ is in fact $L^2([-\pi,\pi], \mu_{\Fm})$, we recall that the Stone-Weierstrass theorem (Fej{\'e}r's theorem) implies that $\mathscr{H}$ is dense in the space of $2\pi$-periodic continuous functions on $[-\pi,\pi]$. Moreover, by Proposition \ref{prop:muFm}, $\mu_{\Fm}$ is a finite Radon measure (i.e., finite and regular) on $[-\pi,\pi]$. The set of continuous functions with compact support are therefore in turn uniformly dense
in $L^2 ([-\pi,\pi], \mu_{\Fm})$ \citep[see e.g.,][]{bog06,R87}. Consequently we find $\bar{\mathscr{H}}=L^2 ([-\pi,\pi], \mu_{\Fm})$. The inverse mapping $\mathcal{T}^{-1}: L^2 ([-\pi,\pi], \mu_{\Fm}) \to \bar{\mathcal{H}}$ is therefore properly defined. This finishes the proof of the first part of the theorem.   \\

Les us then define, for any $\omega \in (-\pi,\pi]$, the process
\[ Z_{\omega} = \mathcal{T}^{-1}\big(1_{(-\pi,\omega]}(\cdot)\big)\] 
with $Z_{-\pi} \equiv 0 \in H$. By the established isometry, this process is well-defined in $\bar{\mathcal{H}}$. Therefore there must exist a sequence $\{Y_n\}$ in $\mathcal{H}$ such that $\lim_{n \to \infty} \|Y_n-Z_{\omega}\|_{\mathbb{H}} = 0$. Since all elements in the sequence have zero-mean, continuity of the inner product implies
\begin{align*}
\innerprod{Z_{\omega}}{f}_\Hspace & 
=  \lim_{n \to \infty} \Tr(\E (Y_n \otimes f))   = \lim_{n \to \infty} \innerprod{\E [Y_n]}{f} =0 \quad \forall f \nequiv 0 \in H 
\end{align*}
showing the process $\{Z_{\omega}: -\pi \le \omega \le \pi\}$ has zero mean. Additionally,
\begin{align*} 
\innerprod{Z_{\omega_4} - Z_{\omega_3}}{Z_{\omega_2} - Z_{\omega_1}}_{\Hspace}  
& 
= \innerprod{1_{(\omega_3,\omega_4]}(\cdot)}{ 1_{(\omega_1,\omega_2]}(\cdot)}_{\mathscr{H}} 
\\& =  \int_{-\pi}^{\pi} 1_{(\omega_3,\omega_4]}(\omega) 1_{(\omega_1,\omega_2]}(\omega) d\mu_{\Fm}(\omega). \label{orhinc} \tageq
\end{align*}
For all $(\omega_1,\omega_2] \cap (\omega_3,\omega_4]=\emptyset$, this inner product is zero while for $\omega_3 = \omega_1, \omega_4=\omega_2$ we have
\[
\innerprod{Z_{\omega_2} - Z_{\omega_1}}{Z_{\omega_2} - Z_{\omega_1}}_{\Hspace} = \mu_{\Fm}(\omega_2) -\mu_{\Fm}(\omega_1).
\qquad \omega_1 \le \omega_2\]
showing that the $\{Z_{\omega}\}$ is right-continuous. We can also write \eqref{orhinc} as 
\[\Tr\big( \int_{-\pi}^{\pi} 1_{(\omega_3,\omega_4]}(\omega) 1_{(\omega_1,\omega_2]}(\omega) d{\Fm}(\omega)\big).\]
For $\omega_3 = \omega_1, \omega_4=\omega_2$ this implies 
\[
\E (Z_{\omega_2} - Z_{\omega_1})\otimes {{(Z_{\omega_2} - Z_{\omega_1})}} = \Fm(\omega_2) -\Fm(\omega_1),
\qquad \omega_1 \le \omega_2.\]
where the equality holds in $\snorm{\cdot}_1$. The second order structure of $Z_{\omega}$ is therefore uniquely defined by the operator-valued measure $\Fm$ of the process $X$. 
\end{proof}
The generalization of the Cram{\'e}r representation to processes of which the spectral density operator is not necessarily well-defined is given in the following theorem.
\begin{theorem}[Functional Cram{\'e}r representation] \label{Cramer}
\textcolor{black}{Let $X$ be a weakly stationary functional time series.} Then there exists a right-continuous functional orthogonal increment process $\{Z_{\omega}, -\pi \le \omega \le \pi\}$ with $Z_{-\pi}\equiv 0 \in H$ such that
\[
X_t = \int_{-\pi}^{\pi} e^{\im t \omega} d Z_{\omega} \qquad \text{a.s.}
\]
\end{theorem}
\begin{proof}[Proof of Theorem \ref{Cramer}]
Consider the subspace $\mathscr{H}_s$ of $\bar{\mathscr{H}}$ containing the simple functions, i.e., the space $\mathscr{H}_s$ contains elements of the form
\[
\textcolor{black}{g(\omega) = \sum_{i=1}^{n} a_i 1_{(\omega_{i},\omega_{i+1}]}}
\]
for a partition $P_n =\{ - \pi = \omega_0 < \omega_1  < \cdots < \omega_{n+1} = \pi\}$ of $[-\pi,\pi]$ and  $a_i \in \cnum$. 
Then define the mapping $\mathcal{I}: \mathscr{H}_s \to \bar{\mathcal{H}}$ given by 
\[
\mathcal{I}(g) = \sum_{i=0}^{n} a_i (Z_{\omega_{i+1}}-Z_{\omega_i}).
\]
By Theorem \ref{isomorph}, this is an isomorphism from $\bar{\mathscr{H}}$ onto $\bar{\mathcal{H}}$ and coincides with $\mathcal{T}^{-1}$. More specifically, $\mathcal{I}(e^{\im \cdot t}) =\mathcal{T}^{-1}(e^{\im \cdot t}) =\mathcal{T}^{-1}\mathcal{T} (X_t)= X_t  $ and the statement of the Theorem follows by taking the Riemann-Stieltjes integral limit 
\[\bignorm{X_t -  \sum_{i=0}^{n} e^{\im \omega_i t} (Z_{\omega_{i+1}}-Z_{\omega_i} )}^2_{\Hspace} \to 0 \text{ as } n \to  \infty,\]
where $\text{mesh}(P_n) \to 0$ as $n \to \infty$. More generally, for any $g \in L^2([-\pi,\pi], \mu_{\Fm})$, the mapping $\mathcal{I}(g)$ corresponds to the Riemann--Stieltjes integral with respect to the orthogonal increment process $Z_{\omega}$. 
 \end{proof}
 
\textcolor{black}{\begin{remark}
\textit{It is worth to mention that, by means of the isometric isomorphism in Theorem \ref{isomorph}, we can write $Z_{\omega}$ also in terms of the limit of a weighted sum of the functions $X_t$. Firstly, we note that the indicator function $1_{(-\pi,\omega]}(\cdot)$ can be approximated in $L^2([-\pi,\pi],\mu_{\Fm})$ by the $N$-th order Fourier series approximation $
b_{N}(\lambda) =\lsum_{|t|\leq N}\tilde{b}_{\omega, t}\,e^{\im t \lambda}$
where the Fourier coefficients are given by $\tilde{b}_{\omega,t} =\frac{1}{2\pi} \int^{\pi}_{-\pi} 1_{(-\pi,\omega]}(\lambda)\,e^{-\im t \lambda}\,d\lambda$. More specifically, we have $\|b_N(\cdot) -1_{(-\pi,\omega]}(\cdot) \|^2_\mathscr{H} \to 0$ as $N \to \infty$ \cite[see e.g.,][]{BrockwellDavis}. Consider therefore the following weighted sum
\begin{align*}
Z^{(N)}_{\omega} = \frac{1}{2\pi}\sum_{|t|\leq N}X_{t}\int^{\pi}_{-\pi} 1_{(-\pi,\omega]}(\lambda)\,e^{\im t \lambda}\,d\lambda=\sum_{|t|\leq N}\tilde{b}_{\omega,t}\,X_{t},
\end{align*} 
Then by the Theorem  \ref{isomorph}
\[
\lim_{N \to \infty}\|Z^{N}_{\omega}-Z_{\omega}\|^2_{\mathbb{H}} = \|\mathcal{T}^{-1}b_N(\cdot) -\mathcal{T}^{-1}1_{(-\pi,\omega]}(\cdot) \|^2_\mathbb{H} = \|b_N(\cdot) -1_{(-\pi,\omega]}(\cdot) \|^2_\mathscr{H}=0.
\]}
 \end{remark}}
 
 
In case of discontinuities in the spectral measure $\Fm$ we can decompose the process into a purely indeterministic component and a purely deterministic component. 
 \begin{proposition}\label{cramer_disc}
Assume the spectral measure $\Fm$ of a \textcolor{black}{weakly stationary functional time series} $\{X_t\}$ has $k$ points of discontinuity at $\omega_1, \ldots, \omega_k$. Then with probability one
\begin{align}\label{eq:cramer_disc}
 X_t = \int_{(-\pi, \pi] \setminus  \{ \omega_1, \ldots, \omega_k\}} e^{\im t \omega} d Z_{\omega} + \sum_{\ell=1}^{k}(Z_{\omega_\ell} - Z_{\omega^{-}_\ell})e^{\im t \omega_\ell},
\end{align}
\textcolor{black}{where $Z_{\omega^{-}_\ell} = \lim_{\omega \uparrow \omega_l} \|Z_{\omega_\ell}-Z_{\omega} \|^2_\mathbb{H}=0$}. Furthermore,   
all terms on the right hand side of \eqref{eq:cramer_disc} are uncorrelated and
\[
\Var(Z_{\omega_\ell} - Z_{\omega^{-}_\ell} ) = \Fm_{\omega_\ell}-\Fm_{\omega^{-}_\ell}
\]
for each $\ell=1,\ldots, k$.
 \end{proposition}
 The proof is relegated to the Appendix. The spectral representation in Proposition \ref{cramer_disc} can be used to define a Cram{\'e}r--Karhunen--Lo{\`e}ve representation for processes of which the spectral measure has finitely many discontinuities. 

\begin{definition}[Cram{\'e}r--Karhunen--Lo{\`e}ve representation] \label{CKL}
Suppose that \textcolor{black}{the weakly stationary functional time series} $X=\{X_t\}$ is given by
\[
X_t = \int_{M} e^{\im t \omega} d Z_{\omega} +\sum_{\ell=1}^{k} (Z_{\omega_\ell} - Z_{\omega^{-}_\ell})e^{\im t \omega_\ell},
\]
where $M=(-\pi, \pi] \setminus  \{ \omega_1,\ldots, \omega_k\}$ for $-\pi<\omega_1<\ldots<\omega_k\leq\pi$, and let $\Fm$ be the spectral measure of $X$. \textcolor{black}{Furthermore, assume there exists an operator-valued function $\omega \mapsto \F_{\omega}$ such that $\Fm(M) = \int_{M} \F_{\omega} d\omega$} 
with eigendecomposition
\[\F_{\omega} = \sum_{j=1}^{\infty} \nu^{\omega}_j\phi^{\omega}_j \otimes \phi^{\omega}_j.\]
Furthermore, let
\[
 \Fm(\omega_\ell)-  \Fm({\omega^{-}_\ell}) = \sum_{j=1}^{\infty} \nu^{\omega_\ell}_j\phi^{\omega_\ell}_j \otimes \phi^{\omega_\ell}_j
\]
be the eigendecomposition of $\Fm(\omega_\ell)-  \Fm({\omega^{-}_\ell})$ for $\ell=1,\ldots,k$. Then, we can write
\[X_t = \int_{(-\pi, \pi] \setminus   \{ \omega_1,\ldots, \omega_k\}} e^{\im t \omega} \big(\lsum_{j=1}^{\infty} \phi^{\omega}_j \otimes \phi^{\omega}_j \big) d Z_{\omega} +\lsum_{\ell=1}^{k} \big(\lsum_{j=1}^{\infty}\phi^{\omega_\ell}_j \otimes \phi^{\omega_\ell}_j\big) (Z_{\omega_\ell} - Z_{\omega^{-}_\ell})e^{\im t \omega_\ell}.
\]
which is the \textit{Cram{\'e}r--Karhunen--Lo{\`e}ve representation} of the process $\{X_t\}$.  
\end{definition}
Note that the spectral measure for all measurable sets $[-\pi,\pi]$ has positive definite increments and therefore $ \Fm(\omega_\ell)-  \Fm({\omega^{-}_\ell})$ has an eigendecomposition with positive eigenvalues. 
If there are no discontinuities then the Cram{\'e}r--Karhunen--Lo{\`e}ve representation simply coincides with the indeterministic component of Definition \ref{CKL}, i.e., 
\[
X_t = \int_{(-\pi, \pi]} e^{\im t \omega} \big(\lsum_{j=1}^{\infty} \phi^{\omega}_j \otimes \phi^{\omega}_j \big) d Z_{\omega} \tageq \label{CKL_con}
\]
In order to derive an optimal finite dimensional representation of the indeterministic component of the process, we require in Definition \ref{CKL} that there is a well-defined spectral density operator except on sets of measure zero. We remark that this assumption also covers a harmonic principal component analysis of  long-memory processes (see Remark \ref{longmem}) and holds under much weaker conditions \citep[see e.g.,][]{Hormann2015} than 
those stated in \cite{Panar2013b}, who originally derived a Cram{\'e}r--Karhunen--Lo{\`e}ve representation of the form \eqref{CKL_con} for processes with short-memory. \\
The Cram{\'e}r--Karhunen--Lo{\`e}ve representation in Definition \ref{CKL} can be seen to encapsulate the full second order dynamics of the process and gives insight into an optimal finite dimensional representation. \textcolor{black}{As originally noted in \cite{Panar2013b}, such a representation can be viewed as a} 
`double' spectral representation in the sense that it first decomposes the process into uncorrelated functional frequency components and in turn provides a spectral decomposition in terms of dimension. This is more easily seen by noting that formally we can write it as
\[X_t = \int_{(-\pi, \pi] \setminus   \{ \omega_1,\ldots, \omega_k\}} e^{\im t \omega} \lsum_{j=1}^{\infty}\inprod{ d Z_{\omega}}{\phi^{\omega}_j} \phi^{\omega}_j +\lsum_{\ell=1}^{k} \lsum_{j=1}^{\infty}\inprod{Z_{\omega_\ell} - Z_{\omega^{-}_\ell}}{\phi^{\omega_\ell}_j}  \phi^{\omega_\ell}_j e^{\im t \omega_\ell}.
\]
 Just like the Karhunen-Lo{\`e}ve representation for independent functional data, it separates the stochastic part from the functional part and provides information on the smoothness of the random curves. Furthermore, it enables to represent each frequency component into an optimal basis where its dimensionality can be derived from the relative contribution of the component to the total variation of the process. A truncation of the infinite sums at a finite level therefore allows an optimal way to construct a finite dimensional representation of the process. 
\\
 Such a truncation for processes that satisfy Definition \ref{CKL} requires that stochastic integrals of the form $\int_{-\pi}^{\pi} U_{\omega}\,dZ_{\omega}$ are well-defined where $U_{\omega}$ is an element of the Bochner space $\mathcal{B}_{\infty} =L^2_{S_{\infty}(H)}([-\pi,\pi],\mu_{\Fm})$ of all strongly measurable functions $U:[-\pi,\pi]\to S_{\infty}(H)$ such that 
\[
\|U\|^2_{\mathcal{B}_\infty} =\int_\Pi \snorm{U_{\omega}}^2_{\infty} d\mu(\omega)<\infty 
\] with
\[
\mu(E)=\int_{E} d \mu_{\Fm}(\omega),
\]
for all Borel sets $E\subseteq[-\pi,\pi]$ and where $\mu_{\Fm}$ is the measure in \eqref{mufm}. This is proved in the Appendix (Proposition  \ref{prop:Stochint}) and generalizes the result in Appendix B 2.3 of \citet{vde16}. With this in place, we obtain a harmonic principal component analysis for processes of which the spectral measure has finitely many jumps. 
\begin{corollary}[Harmonic functional principal component analysis] \label{cor:Hfpca}
Suppose $\{X_t\}$ has a Cram{\'e}r--Karhunen--Lo{\`e}ve representation as in Definition \ref{CKL}. Then,
\textcolor{black}{for any $p: [-\pi,\pi] \to \mathbb{N}$ c{\`a}dl{\`a}g, the random function
\[
X^{\star}_t = \int_{(-\pi, \pi] \setminus  \{ \omega_1,\ldots, \omega_k\}}  e^{\im \omega t} \Big(\sum_{j=1}^{p(\omega)} \phi^{\omega}_j \otimes \phi^{\omega}_j \Big) dZ_{\omega} +\lsum_{\ell=1}^{k} \big(\lsum_{j=1}^{p(\omega_\ell)}\phi^{\omega_\ell}_j \otimes \phi^{\omega_\ell}_j\big) (Z_{\omega_\ell} - Z_{\omega^{-}_\ell})e^{\im t \omega_\ell}.
\]
minimizes the mean squared error among all linear rank reductions of $\{X_t\}$ to a process $\{Y_t\}$ with representation $Y_t = \int_{-\pi}^{\pi} e^{\im \omega t} A_{\omega} d\omega$ where $A_\omega \in \mathcal{B}_{\infty}$ with $\text{rank}(A_{\omega}) \le p(\omega)$, i.e., 
\[
\|X_t - X^{\star}_t\|^2_{\mathbb{H}} \le 
\|X_t -Y_t\|^2_{\mathbb{H}}
\]
subject to the constraint $\text{rank}(A_{\omega}) \le p(\omega)$.
}
\textcolor{black}{The minimized error is given by
\[
\|X_t-X^{\star}_t\|^2_\mathbb{H} =\int_{(-\pi, \pi] \setminus \{\omega_1,\ldots, \omega_k\}} \big(\lsum_{j>p(\omega)} \nu^{\omega}_j \big)  d\omega + \lsum_{\ell=1}^{k}\big(\lsum_{j>p(\omega_\ell)} \nu^{\omega_\ell}_j \big).
\]}
\end{corollary}
\textcolor{black}{Note that the rank of $A_{\omega}$ constraints the dimensions of $H$-valued processes with representation $\int_{-\pi}^{\pi} e^{\im \omega t} A_{\omega} d\omega$ to a lower-dimensional subspace of $H$ \citep[see also][]{Panar2013b}. }
\begin{proof}
Without loss of generality, we prove this for the case of one discontinuity at frequency $\omega_o$. By orthogonality of the two parts the representation, we find using Proposition \ref{cramer_disc} and Fubini's theorem
\begin{align*}
\|X_t-Y_t\|^2_\mathbb{H} & =\Bignorm{\int_{(-\pi, \pi] \setminus  \{ \omega_o\}}\big(I-A_{\omega}\big) e^{\im \omega t}dZ_{\omega}}^2_\mathbb{H} + \Bignorm{\big(I-A_{\omega_o}\big) (Z_{\omega_o} - Z_{\omega^{-}_o})e^{\im t \omega_o}}^2_\mathbb{H}
\\ &= \int_{-\pi}^{\pi}  \Tr\big(   \big(I- A_{\omega} \big) \F_{\omega}d\omega\big(I- A_{\omega} \big)^\dagger \big) +  \Tr\big(   \big(I- A_{\omega_o} \big) \big(\Fm_{\omega_o}-\Fm_{\omega^{-}_o}\big)\big(I- A_{\omega} \big)^\dagger \big)
\end{align*}
From which it is straightforward to see that this is minimized $X^{\star}_t$ where error is given by 
\begin{align*}
\|X_t-X^{\star}_t\|^2_\mathbb{H}&= \Bignorm{\int_{(-\pi, \pi] \setminus  \{ \omega_o\}}   e^{\im \omega t} \big(\lsum_{j>p(\omega)} \phi^{\omega}_j \otimes \phi^{\omega}_j \big) dZ_{\omega}}^2_\mathbb{H} 
 +\Bignorm{\big(\lsum_{j> p(\omega_o)}\phi^{\omega_o}_j \otimes \phi^{\omega_o}_j\big) (Z_{\omega_o} - Z_{\omega^{-}_o})e^{\im t \omega_o}}^2_\mathbb{H} 
\\& =\textcolor{black}{ \int_{(-\pi, \pi] \setminus  \{ \omega_o\}} \big(\lsum_{j>p(\omega)} \nu^{\omega}_j \big)  d\omega + \big(\lsum_{j>p(\omega_o)} \nu^{\omega_o}_j \big).}
\end{align*}
\end{proof}

\begin{remark}[Harmonic functional principal component analysis of long-\\memory processes] \label{longmem}
In analogy to classical time series, the covariance structure of a long-memory functional time series does not decay rapidly. Without loss of generality, assume such a process will have its covariance structure satisfy
\[
\mathcal{C}_h \sim B h^{2d-1} \quad 0 < d < 0.5,
\]
where $B$ a strictly positive element of $S_1^+(H)$. It is clear that for such a process, the dependence structure does not decay rapidly enough for $\sum_{h \in \mathbb{Z}}\snorm{C_h}_p =\infty$ to hold. In order to understand what can be said about the properties of the spectral density operator, note that we can for simplicity mimic the behavior of such process by considering the linear process 
\begin{align*}
X_t = \sum_{j=0}^{\infty} \big(\lprod_{0 < k \le j}\textstyle{\frac{k-1+d}{k}}\big) \varepsilon_{t-j}
\end{align*}
where $\varepsilon_t$ is $H$-valued white noise  and hence by Theorem \ref{Herglotzthm}, the second order structure is given by
$\mathcal{C}^{\varepsilon}_0=\int_{-\pi}^{\pi} d \Fm^{\varepsilon}(\omega) =2 \pi \F^{\varepsilon}_0$.
Using the properties of the Gamma function, a standard argument shows that the filter applied to $\{\epsilon_t\}$ yields
\begin{align*}
C^{X}_h =\int_{-\pi}^{\pi} e^{\im \omega h} d \Fm^{X}(\omega) =\int_{-\pi}^{\pi} e^{\im \omega h} (1-e^{-\im \omega} )^{-2d} d \Fm^{\varepsilon}(\omega) 
\end{align*}
and hence a density of the spectral measure $\Fm^{X}$ at $\omega=0$ for $d>0$ is not defined. Yet, since this has measure $0$, we can, under the conditions of Theorem \ref{Cramer}, define a harmonic principal component analysis as in Corollary \ref{cor:Hfpca} where the number of discontinuities is $k=0$. That is, the optimal approximating process is given by 
\[
X^{\star}_t = \int_{(-\pi, \pi]}   e^{\im \omega t} \Big(\sum_{j=1}^{p(\omega)} \phi^{\omega}_j \otimes \phi^{\omega}_j \Big) dZ_{\omega} 
\]
and the minimized error is given by $\|X_t-X^{\star}_t\|^2_\mathbb{H} =\int_{(-\pi, \pi] } \big(\lsum_{j>p(\omega)} \nu^{\omega}_j \big)  d\omega.$
\end{remark}

\medskip
\noindent 
{\bf Acknowledgements.}
This work has been supported in part by the Collaborative Research Center ``Statistical modeling of nonlinear dynamic processes'' (SFB 823, Project A1, C1, A7) of the German Research Foundation (DFG).


\appendix
\setcounter{section}{0}
\setcounter{equation}{0}
\def\theequation{A\arabic{section}.\arabic{equation}}
\def\thesection{A\arabic{section}}
\section{Some background material} \label{ap:backgr}

In this section, we collect some definitions and concepts which are needed to formalize the construction of the measure, which was heuristically described in Section \ref{Herglotz} and which is described in more detail in Appendix \ref{sec:proofssec3}. 
For more background on topology (on operator algebras) we refer to, e.g., \cite{Munk,KadRing1997I,KadRing1997II,Erdman2015} and for functional analysis to \cite{Conway1990,R91}.

\begin{definition}[monoid] \label{def:monoid}
A set $M$ together with a binary operation $\circ$ is called a \text{monoid}, $(M, \circ)$ if the following axioms are satisfied
\begin{romanlist}
\item[] (closure) $\phantom{ivity}$ for any $a,b \in M$, $a \circ b \in M$;
\item[] (associativity) for any $a,b,c \in M$, we have $(a \circ b) \circ c = a \circ (b \circ c)$;
\item[] (identity) $\phantom{vity}$ there exists an $e \in M$ such that, for any $a \in M$, $e \circ a = a \circ e = a$.
\end{romanlist}
\end{definition}

\begin{definition}[monoid homomorphism]
Let $(M,\star)$ and $(N,\circ)$ be two monoids with identity elements $e_M$ and $e_N$, respectively. Then a function $f:(M,\star) \to (N,\circ)$ is a \textit{monoid homomorphism} if it preserves the monoid operation and identity element, i.e., $f(m \star n) =f(m) \circ f(n)$, for all $m,n \in M$ and $f(e_M) = e_N$.
\end{definition}

\begin{definition}[topological monoid]
A topological monoid is a monoid $(M,\circ)$ endowed with a topology $\tau$ such that the binary operation $\circ: M\times M \to M$ is continuous.
\end{definition}

For two topological monoids $(X,\circ)$ and $(Y,\star)$, we denote $\text{Hom}_{\mathrm{mon}}(X,Y)$ as the set of all monoid homomorphisms $X \to Y$.

\begin{definition}[initial topology] \label{def:inittop}
Given a set $X$ and an indexed family of topological spaces $(Y_i)_{i \in I}$ with functions $f_i : X \to Y_i$. Then the \textit{initial topology on $X$}, $\tau_{\text{in}}$, induced by the functions $(f_i)_{i \in I}$ is the coarsest topology on $X$ s.t. each 
\[f_i : (X,\tau_{\text{in}}) \to Y_i\] is continuous.
Examples of initial topologies used in Section \ref{sec:proofssec3}:
\begin{enumerate}[leftmargin=*]
\item[(i)]{\em \bf Cartesian product endowed with the product topology $\tau_p$}:\\ Let $(Y_x)_{x \in X}$ be an indexed family of sets. The \textit{cartesian product}, denoted by $\Pi_{x \in X} Y_x$ is the set of functions $f: X \to  \bigcup Y_x$ such that $f(x) \in Y_x$ for each $x \in X$. The maps  $\pi_{x_1}: \Pi_{x \in X} Y_x \to Y_{x_1}: f \mapsto f(x_1)$, $x_1 \in X$, are called the \textit{canonical coordinate projections}. The \textit{product topology} on $\Pi_{x \in X} Y_x$ is the initial topology on  $\Pi_{x \in X} Y_x$ induced by the projection maps $\pi_x, x\in X$. 
\item[(ii)]{\em \bf Sets of all functions $(Y)^X$ with the product topology  $\tau_p$}:\\ 
Let $X$ and $Y$ be sets. We denote $(Y)^X$ the \textit{set of all functions from $X$ to $Y$}. We remark that by Tychonoff's theorem, this set is compact if $Y$ is compact. The set $(Y)^X$ is the cartesian product where $Y_x = Y$ for every $x \in X$. In this case, the coordinate projections become \textit{evaluation maps}, i.e., for each $x_1 \in X$, the map $\pi_{x_1}:(Y)^X\to Y$ takes each point $f \in (Y)^X$ to its value, i.e., $\pi_{x_1}(f) = f(x_1)$. The product topology on $Y$ is known as the \textit{topology of pointwise convergence} since, when endowed with this topology,  a net of functions $(f_\alpha)$ in $(Y^X, \tau_p)$ converges to a function $f \in Y^X$ if and only if $f_\alpha(x) \to f(x)$ in $Y$ for every $x \in X$. 
\item[(iii)]{ \em \bf Subsets endowed with the subspace topology}:\\ Let $(Y,\tau)$ be a topological space and let $Y_o \subset Y$ be a subset. \textit{\em The subspace topology}, $\tau_{Y_o}$, on $Y_o$ is the initial topology with respect to the inclusion map $i: Y_o \to Y$, i.e., the map  $i(y) = y$ for all $y \in Y_o$. 
The topological space $(Y_o,\tau_{Y_o})$, is called a \textit{topological subspace} of $(Y,\tau)$. $Y_o$ is a closed subspace of $Y$ if $Y\setminus Y_o \in \tau$ (i.e., the complement is open). A closed subspace of a compact topological space is compact \citep[see e.g.][]{Munk}. 
\end{enumerate}
\end{definition}

\section{\textcolor{black}{On the construction of the measure}} \label{sec:proofssec3}

In this section we provide some more detail of the construction of the measure, which was described heuristically in Section \ref{Herglotz}. The construction is achieved by means of an embedding of the cone into its bidual cone $\phi: \lo(H)^+ \to (\lo(H)^+)^{\prime \prime}$, where $(\lo(H)^+)^{\prime \prime}$ consists of all positive continuous linear functionals on the dual cone $(\lo(H)^+)^{\prime}$. To make this more precise, we mention some properties of cones.
\subsection{Dual pair of cones} \label{dualcone}
\begin{definition}[Cones and some properties]
A subset $C$ of a real vector space $V$ is called a \textit{cone} it is closed under positive scalar multiplication, i.e., if $\lambda \in \mathbb{R}$, $c \in C$, then $\lambda c \in C$. It is called a \textit{convex cone} if it is moreover closed under linear combinations, i.e., $\lambda, \beta \in \mathbb{R}$, $c_1, c_2 \in C$, then $\lambda c_1 +\beta c_2 \in C$. 
A cone is called \textit{pointed} if it contains the zero element, i.e., if it satisfies $C \cap -C =\{0\}$, while it is called \textit{generating} if $C-C = V$.
\end{definition}
We moreover need the notion of a topological dual cone, which we shall simply refer to as the dual cone in the subsequent sections. 
\begin{definition}[Topological dual cone] 
The \textit{topological dual cone} $(C)^\prime$ to a cone $C$ of a real vector space $V$ can be defined as the set
 \[\{v \in V^{'} : \innerprod{v}{c} \ge 0 \quad \forall c \in C \}\]
where $\innerprod{\cdot}{\cdot}$ denotes the duality pairing between $V$ and its topological dual $V^{'}$.
\end{definition}

\subsection{\textcolor{black}{The embedding}} \label{sec:apem}

Given the duality pairing in \eqref{eq:pairapp} of the underlying spaces $\lo(H)^\dagger$ and $S_1(H)^\dagger$,  we can identify the cone $S_1(H)^+$ as the dual cone of $\lo(H)^+$.
\begin{proposition}\label{prop:dualcones}
The dual cone of $\lo(H)^+$, $(\lo(H)^+)^\prime $, is given by
\[
S_1^+(H)=\{ A \in S_1^\dagger(H): \tr(AB) \ge 0 \quad \forall B \in \lo(H)^+\}.
\]
\end{proposition} 
\begin{proof}[Proof of Proposition \ref{prop:dualcones}] We start by remarking that positive elements on a $C^\star$- algebra form a closed convex cone $\{a^\star a : a \in C^\star\}$\citep[see e.g.][]{Conway1990}. Since elements of $\lo(H)^+$ have a positive square root, $\lo(H)^+$ forms a pointed convex cone in the Banach algebra of bounded linear operators. By the spectral theorem it is moreover generating the space $\lo(H)^\dagger$. With a similar argument, it is straightforward to verify that $C_{V} = S_1(H)^+$ is a pointed convex cone in the Banach space $V =(S_{1}(H)^{\dagger}, \snorm{\cdot}_{1})$. Its topological dual is the Banach space $V^{'} = (\lo(H)^{\dagger}, \snorm{\cdot}_{\lo})$ with cone $C_{V^{'}}: = (\lo(H)^{+}, \snorm{\cdot}_{\lo})$ in $V^{'}$. 
We can naturally identify an element of a Banach space with an element from its bidual by means of a canonical embedding \citep[see e.g.][]{KadRing1997I}. More specifically, we have that $V$ canonically embeds into $ {(V^{'})}^{'}$ i.e., $V \subseteq {(V^{'})}^{'} \rightarrow C_{V} \subseteq (C_{V^{'}})^\prime$ . Since the cone $C_{V^{'}}$ is closed, the Hahn-Banach separation theorem implies that $C_V =(C_{V^{'}})^{\prime}$, which identifies the cone of non-negative trace class operators as the dual cone of the non-negative bounded linear operators. 
\end{proof}
To identify $\lo(H)^+$ with its image into its bidual cone $(\lo(H)^+)^{\prime \prime}$, we consider an injection $B \mapsto (A \mapsto \tr(AB))$, $A \in S_1(H)^+, B \in \lo(H)^+$. 
The image of the mapping $\phi$ is therefore simply given by
\[
\phi(\lo(H)^+=\{\phi_B\in(\rnum^+)^{S_1(H)^+} : B\in\lo(H)^+\}\tageq \label{eq:top}
\]
where $\phi_B:S_1(H)^+\to\rnum^+$ is given by $\phi_B(A)=\tr(BA)\geq 0$.The following result is essential in order to uniquely identify $\lo(H)^+$ with a family of $\rnum^+$-valued measures and to use $\phi(\lo(H)^+)$ in order to construct the compactification of $\lo(H)^+$.
\begin{proposition}\label{prop:fact}
The set $\phi(\lo(H)^+)$ coincides with the set of all positive continuous linear functionals $S_1^+(H) \to \rnum^+$, that is, $\phi(\lo(H)^+)= (S_1^+(H))^\prime$.
\end{proposition}
\begin{proof}[Proof of Proposition \ref{prop:fact}]
Recall that $\lo(H)^\dagger$ is the topological dual space of $S_1(H)^\dagger$ where the pairing is given by \eqref{eq:pairapp}. A functional on a Banach space $V$ is said to be positive if $f(v) \ge 0$ for all non-negative elements $v \in V$. It can be shown that all positive functionals on a Banach space are continuous  \citep[][Prop I.7]{Neeb98}. Since $(S_1(H)^\dagger,\snorm{\cdot}_1)$ is a Banach space this implies that all $f: S_1(H)^+ \to \rnum^+$ must be continuous. But since $\lo(H)^\dagger$ consists of all continuous linear functionals on $S_1(H)^\dagger$, all continuous positive functionals on $S_1(H)^+$ must be of the form \eqref{eq:top}.
\end{proof}
\textcolor{black}{
Given this is in place, we can now obtain that $\phi$ provides an isomorphism between $\lo(H)^+$ and $\operatorname{Hom}_{mon}(S_1(H)^+,\rnum^+)$, where the isometry property of the spaces of self-adjoint operators continues to hold when restricted to the cones.
\begin{theorem}\label{prop:fact2} 
Let $\operatorname{Hom}_{mon}(S_1(H)^+,\rnum^{+})$ denote the set of all monoid homorphisms $S_1(H)^+ \to \rnum_{+}$. Then $\phi$ in \eqref{eq:top} is an isometric isomorphism between $\lo(H)^+$ and $ \operatorname{Hom}_{mon}(S_1(H)^+,\rnum^+)$, i.e.,
\[ \phi(\lo(H)^+) = \operatorname{Hom}_{mon}(S_1(H)^+,\rnum^+) \cong_{m} \lo(H)^+.  \tageq \label{eq:isom}\]
\end{theorem}}
\begin{proof}
While it is immediate the mapping is injective, note that Proposition \ref{prop:fact} implies that for any linear functional $\varphi: S_1(H)^+ \to \rnum^+$ there exists an element $B \in \lo(H)^+$, such that $\phi(B)=\varphi$. We therefore have established that $\phi$ is a bijective map between $\lo(H)^+$ and $(S_1(H)^+)^\prime$. It is easily verified using the definition (see Appendix \ref{ap:backgr}) that the cones $\lo(H)^+, S_1(H)^+, \rnum^+$, which are not vector spaces, have the algebraic structure of topological monoids with respect to addition. It is moreover straightforward to verify that all functionals in $(S_1(H)^+)^\prime$ preserve the monoid structure between $S_1(H)^+$ and $\rnum^+$. Indeed, since any element in $(S_1(H)^+)^\prime$ is given by a functional of the form $\phi_B: S_1(H)^+ \to \rnum^+, A\mapsto \tr(BA), B \in \lo(H)^+$, we obtain from linearity of the trace that $\phi_B(A_1+A_2) = \phi_B(A_1)+\phi_B(A_2)$, while $\phi_B(O_H) = 0$, for any $B \in \lo(H)^+$ and $A_1, A_2 \in S_1(H)^+$. By Proposition \ref{prop:fact}, $\phi(\lo(H)^+)$ consists of all such monoid homomorphisms, i.e., $\phi(\lo(H)^+) = \operatorname{Hom}_{mon}(S_1(H)^+,\rnum_{+})$, where $\phi(\lo(H)^+)$ has itself a monoid structure, since the axioms of closure, associativity and identity are easily checked (See Definition \ref{def:monoid}). Hence, we find that $\phi$ is an isomorphism between the monoids $\lo(H)^+$ and $\operatorname{Hom}_{mon}(S_1(H)^+,\rnum_{+})$. It remains to verify the isometry property $\|B\|_{\infty}=\|\phi(B)\|$ for elements restricted to the cone $\lo(H)^+$. For completeness, we provide the argument. Note that since $\lo(H)^+$ generates $\lo(H)^\dagger$, the norm $\|B\|_{\infty}$ is trivially the same in $\lo(H)^+$ and $\lo(H)^\dagger$. Moreover, $\|\phi(B)\| = \sup\{ |\tr(BA)|: \snorm{A}=1, A \in S_1(H)^\dagger\} = \sup\{ |\tr(BA)|: \snorm{A}=1, A \in S_1(H)^+\}$. Now, it is immediate that $\|\phi(B)\| \le \|B\|_{\infty} \snorm{A}_1 \le \|B\|_{\infty}$. In order to show the reverse inequality, note that for nonzero elements in $B \in \lo(H)^\dagger$, we can write the operator norm as $\|B\|_{\infty}= \sup\{ |\innerprod{Bx}{y}|: \|y\|=\|x\|=1\}$, where we set $y = Bx/\|Bx\|$. From this we obtain,
\[ \|B\|_{\infty}= \sup_{\substack{\|y\|=\|x\|=1}}|\innerprod{Bx}{y}| \le  \sup_{\substack{A \in S_1(H)^+}, \snorm{A}=1}|\tr(BA)| = \|\phi(B)\|,\]
for all $B \in \lo(H)^+$. The result now follows.
\end{proof}
\subsection{\textcolor{black}{Compactification of $\lo(H)^+$-valued measures }} \label{sec:aphomcom}
\begin{proposition} \label{prop:homcom}
The set $\operatorname{Hom}_{mon}(S_1(H)^+, \rnum^+_{\infty})$ is a compact topological monoid. Furthermore,  addition is continous with respect to the topology of pointwise convergence inherited from ${(\rnum^+_\infty)}^{S_1(H)^+}$. This coincides with the ultraweak operator topology on $\lo(H)^+$.
\end{proposition}
\begin{proof}
We shall make use of some concepts which are collected in Appendix \ref{ap:backgr}. We start by showing that the set  $\operatorname{Hom}_{mon}(S_1(H)^+, \rnum^+_{\infty})$ is a compact topological monoid of monoid homomorphisms from $S_1(H)^+$ into $\rnum^+_{\infty}$. More specifically, it consists of all positive continuous linear functionals $S_1(H)^+ \to \rnum^+_{\infty}$ by Proposition \ref{prop:fact}. To see that this is a compact  topological  monoid, recall that $S_1(H)^+, \rnum^+$ are additive topological monoids, i.e., addition is continuous and the sets admit a zero element $O_H$, and $0$, respectively. Similarly, the compact set $\rnum^+_{\infty}$ can be viewed as a compact additive topological monoid, where $\infty$ is also treated as a zero element, i.e., for all $x \in \rnum^+_{\infty}, x +\infty = \infty+x = \infty$. 
The set 
\[\text{Hom}_{mon}(S_{1}(H)^+, \rnum^+_{\infty}) \tageq \label{eq:hommonap}\]
is closed when viewed as a subset of the set ${(\rnum^+_\infty)}^{S_1(H)^+}$, the set of all functions from $S_1(H)^+$ into $\rnum^+_{\infty}$. By Tychonoff's theorem, ${(\rnum^+_\infty)}^{S_1(H)^+}$ is compact since $\rnum^+_\infty$ is compact. Being a closed subset of a compact topological space, it is itself compact and the result now follows.
\\
For the second part, note that the initial topology on ${(\rnum^+_\infty)}^{S_1(H)^+}$, implies that the evaluation mappings $\pi_A(f) = f(A), f \in {(\rnum^+_\infty)}^{S_1(H)^+}$ for all $A \in S_1(H)^+$ are continuous. Being a subset of ${(\rnum^+_\infty)}^{S_1(H)^+}$ this implies for all the functionals $\phi_B$ in \eqref{eq:hommonap} that $\pi_A(\phi_B) = \phi_B(A) = \tr(BA), $ are continuous for all $A \in S_1(H)^+$. Hence, we inherit the topology of pointwise convergence. Addition is therefore continuous w.r.t. to this topology. Note that due to the form of the functionals $\phi_B$, this notion of convergence coincides exactly with convergence in the ultraweak topology as given in Definition \ref{def:uw}.
\end{proof}

\section{Proofs of section \ref{Gcram}}

\begin{proposition}\label{prop:Stochint}
Define the Bochner space $\mathcal{B}_{\infty} =L^2_{S_{\infty}(H)}([-\pi,\pi],\mu_{\Fm})$ of all strongly measurable functions $U:[-\pi,\pi]\to S_{\infty}(H)$ such that 
\[
\|U\|^2_{\mathcal{B}_\infty} =\int_\Pi \snorm{U_{\omega}}^2_{\infty} d\mu(\omega)<\infty.   \tageq \label{eq:Bnorm}
\] with
\[
\mu_\Fm(E)=\int_{E} d \mu_{\Fm}(\omega), \tageq \label{eq:measure}
\]
for all Borel sets $E\subseteq[-\pi,\pi]$. Then, for $U \in \mathcal{B}_{\infty}$, the integral
\[\int_{-\pi}^{\pi} U_{\omega}\,dZ_{\omega}\]
exists and belongs to $H$.
\end{proposition}
\begin{proof}
The Proposition follows directly from section B 2.3 of \citet{vde16} by replacing Lemma B 2.5 of the corresponding paper with the following auxiliary lemma.
\end{proof}
 \begin{lemma}\label{bochnerlemma}
Let $X_t$ be a functional process with spectral representation $X_t= \int_{-\pi}^{\pi} e^{\mathrm{i}\omega t} d Z_{\omega}$ for some functional orthogonal increment process $Z_{\omega}$ that satisfies $ \E\Tr(Z_{\omega} \otimes Z_{\omega}) = \int_{-\pi}^{\omega} d \mu_{\Fm}(\alpha)$. Then for $U_1, U_2 \in S_{\infty}(H_\cnum)$ and $\alpha,\beta\in[-\pi,\pi]$
\begin{align}
\tag{$i$} \langle U_1Z_{\alpha}, U_2 Z_{\beta}\rangle_{\mathbb{H}} &=\Tr\Big(U_1 \Big[\int_{-\pi}^{\alpha\wedge\beta}d {\Fm}(\omega) \Big] U_2^{\dagger}\Big)
\intertext{and}
\tag{$ii$} \|U_1 Z_{\alpha}\|^2_{\mathbb{H}} &\le \snorm{U_1}^2_{\infty}   \int_{-\pi}^{\alpha}d \mu_{\Fm}(\omega). 
\end{align}
Consequently, for $\omega_1 > \omega_2 \ge \omega_3 > \omega_4$
\[
\inprod{U_1(Z_{\omega_1}-Z_{\omega_2})}{U_2(Z_{\omega_3}-Z_{\omega_4})}_{\mathbb{H}}=0
\]
and
\[
 \|U_1(Z_{\omega_1}-Z_{\omega_2})\|^2_{\mathbb{H}} \le \snorm{U_1}^2_{\infty} \Tr(\Fm(\omega_2)-\Fm(\omega_1))= \snorm{U_1}^2_{\infty} \int_{\omega_1}^{\omega_2} \mu_{\Fm}(\omega).
\]
\end{lemma}
\begin{proof}
Using \eqref{eq:inTr} and the invariance of the trace under cyclical permutations
\begin{align*}
 \inprod{U_1Z_{\alpha}}{U_2 Z_{\beta}}_{\mathbb{H}} & = \E \inprod{U_2^\dagger U_1 Z_{\alpha}}{Z_{\beta}}\\&
=  \E \Tr\big(U_2^\dagger U_1 ( Z_{\alpha} \otimes Z_{\beta}) \big) \\& 
=\E \Tr\big( U_1 ( Z_{\alpha} \otimes Z_{\beta}) U_2^\dagger \big) \\&
=\Tr\big( U_1 \Big[\int_{-\pi}^{\alpha\wedge\beta}d \Fm(\omega) \Big]  U_2^\dagger \big). 
\end{align*}
Secondly, we note that by Cauchy-Schwarz inequality and \eqref{eq:inTr}
\begin{align*} 
 \langle U_1Z_{\alpha}, U_2 Z_{\beta}\rangle_{\mathbb{H}} \le
&
\le \snorm{U_1}_{\infty} \snorm{U_2}_{\infty} \E\|Z_{\alpha}\|_2\|Z_{\beta}\|_2\\&
\le \snorm{U_1}_{\infty} \snorm{U_2}_{\infty} \E\Tr( Z_{\alpha \wedge \beta} 
\otimes  Z_{\alpha \wedge \beta})
\\& 
\le \snorm{U_1}_{\infty} \snorm{U_2}_{\infty}   \int_{-\pi}^{\alpha\wedge\beta}d \mu_{\Fm}(\omega) < \infty. 
\end{align*}
\end{proof}

 \begin{proof}[Proof of Proposition \ref{cramer_disc}]
 We prove the case for one discontinuity at $\omega_o$ as the argument for finitely many discontinuities is similar. 
 First we remark that the left limit $Z_{\omega^{-}_o}$ is well-defined in $\mathbb{H}$ for any non-decreasing sequence $\{\omega_n\} \uparrow \omega_o$. The limit exists because $\|Z_{\omega_m}-Z_{\omega_n} \|^2_{\mathbb{H}} =  |\mu_{F}({\omega_m})-\mu_{F}({\omega_n})| \to 0$ as $m,n \to \infty$ and the limit is unique for all $\{\nu_n\} \uparrow \omega_o$ since $\|Z_{\nu_n}-Z_{\omega_n} \|^2_{\mathbb{H}} =  |\mu_{F}({\nu_n})-\mu_{F}({\omega_n})| \to 0$ as $n \to \infty$.
 Under the conditions of Theorem \ref{Cramer}, $\{X_t\}$ has a well-defined spectral representation, which can alternatively be written as
\begin{align}\label{specdecom}
X_t = \int_{(-\pi, \pi] \setminus  ( \omega_o-\delta, \omega_o+\delta]} e^{\im t \omega} d Z_{\omega} +\int_{ ( \omega_o-\delta, \omega_o\delta]} e^{\im t \omega} d Z_{\omega},
\end{align}
for $0<\delta < \pi -|\omega_o|$. It can be directly observed that, by orthogonality of these two integrals and continuity of the inner product that the mean square limit of these two terms must be orthogonal. We can therefore treat their respective limits separately. It is straightforward to see that
\[
\| e^{\im t \cdot} 1_{(-\pi, \pi] \setminus  ( \omega_o-\delta, \omega_o+\delta]}(\cdot) -  e^{\im t \cdot} 1_{(-\pi, \pi ]\setminus  \{\omega_o\}}(\cdot) \|_{ \bar{\mathcal{H}}} \to 0 \text{ as } \delta \to 0,
\]
and thus for the first term of \eqref{specdecom}, we find 
\[
\Bignorm{\int_{(-\pi, \pi] \setminus  ( \omega_o-\delta, \omega_o+\delta]} e^{\im t \omega} d Z_{\omega} - \int_{(-\pi, \pi] \setminus  \{\omega_o\} } e^{\im t \omega} d Z_{\omega} }^2_{\mathbb{H}} \to 0 \text{ as } \delta \to 0.
\]
For the second term in \eqref{specdecom}, Minkowski's inquality implies
\begin{align*}
& \Bignorm{ \int_{ ( \omega_o-\delta, \omega_o+\delta]} e^{\im t \omega} d Z_{\omega} -(Z_{\omega_o} - Z_{\omega^{-}_o})e^{\im t \omega_o} }^2_{\mathbb{H}} 
\\& \le
\Bignorm{ \int_{( -\pi, \pi]} e^{\im t (\omega-\omega_o)} 1_{( \omega_o-\delta, \omega_o+\delta]}(\omega) d Z_{\omega}}^2_{\mathbb{H}}  +
\Bignorm{(Z_{\omega_o+\delta} - Z_{\omega_o-\delta})e^{\im t \omega_1} - (Z_{\omega_o} - Z_{\omega^{-}_o})e^{\im t \omega_o} }^2_{\mathbb{H}}. \tageq \label{decompose2}
\end{align*}
For the first term in  \eqref{decompose2}, the isometric mapping established in Theorem \ref{isomorph} together with continuity of the function $e^{\im t \cdot}$ on $\mathbb{R}$ imply
\begin{align*}
\Bignorm{ &\int_{( -\pi, \pi]} e^{\im t (\omega-\omega_o)}  1_{( \omega_o-\delta, \omega_o+\delta]}(\omega) d Z_{\omega}}^2_{\mathbb{H}} \\& \le
\Big({\sup_{ \omega_o-\delta\le \omega\le \omega_o+\delta}|e^{\im t (\omega-\omega_o)}| \big[\mu_{\Fm}(\omega_o+\delta)-\mu_{\Fm}(\omega_o-\delta) \big] \Big)}^{1/2} \to 0 \text{ as } \delta \to 0.
\end{align*}
By right-continuity of the functional-valued increment process $\{Z_{\omega}\}$, the second term in  \eqref{decompose2} converges to $0$ as $\delta \to 0$. Hence, with probability one,
\begin{align}
 X_t = \int_{(-\pi, \pi] \setminus  \{ \omega_o\}} e^{\im t \omega} d Z_{\omega} + (Z_{\omega_o} - Z_{\omega^{-}_o})e^{\im t \omega_o}.
\end{align}
Finally, since the left limit $Z_{\omega^{-}_o}$ is well-defined, Theorem \ref{isomorph} implies
\[
\Var(Z_{\omega_o}-Z_{\omega^{-}_o})=\lim_{\omega_n \uparrow \omega_o}\mathbb{E}\big[(Z_{\omega_o}-Z_{\omega_n}) \otimes( Z_{\omega_o}-Z_{\omega_n})\big]= \Fm(\omega_o)-  \Fm({\omega^{-}_o}).\]
\end{proof}


\end{document}